\documentclass{article}

\usepackage{amsmath,amsfonts,amssymb,url,theorem}

\usepackage[utf8]{inputenc}

{\theorembodyfont{\slshape}
\newtheorem{proposition}{Proposition}[section]
\newtheorem{lemma}[proposition]{Lemma}
\newtheorem{corollary}[proposition]{Corollary}

\newtheorem{theorem}[proposition]{Theorem}}

{\theorembodyfont{\upshape}
\newtheorem{remark}[proposition]{Remark}
\newtheorem{example}[proposition]{Example}
}

\newcommand{\eps}{\varepsilon}

\newcommand{\bbC}{{\mathbb C}}
\newcommand\N{{\mathbb N}}
\renewcommand\P{{\mathbb P}}
\newcommand\Q{{\mathbb Q}}
\newcommand\R{{\mathbb R}}
\newcommand\Z{{\mathbb Z}}

\newcommand{\calA}{{\mathcal{A}}}
\newcommand{\calC}{{\mathcal{C}}}
\newcommand{\DD}{\mathcal{D}}

\newcommand{\norm}{{\mathcal{N}}}
\newcommand{\OO}{{\mathcal{O}}}

\newcommand{\XX}{{\mathcal{X}}}

\newcommand{\tilE}{{\widetilde E}}
\newcommand{\tilF}{{\widetilde F}}

\newcommand{\tilR}{{\widetilde R}}

\newcommand\tilu{{\tilde u}}

\newcommand\qed{\hfill$\square$}

\newcommand{\Mu}{{\mathrm{M}\,}}

\newcommand{\Height}{{\mathrm{H}}}

\newcommand{\Log}{{\mathrm{Log}\,}}

\newcommand{\Spec}{{\mathrm{Spec}\,}}

\newcommand{\sh}{{\mathrm{sh}}}
\newcommand{\nr}{{\mathrm{nr}}}

\newcommand{\gera}{{\mathfrak{a}}}
\newcommand{\gerb}{{\mathfrak{b}}}
\newcommand{\gerp}{{\mathfrak{p}}}

\newcommand{\one}{{\mathbf 1}}
\newcommand{\bfg}{{\mathbf g}}
\newcommand{\bfx}{{\mathbf x}}

\newcommand{\size}[1]{{\overline{|#1|}}{}}
\newcommand{\lsize}[1]{{\underline{|#1|}}}

\DeclareMathOperator{\Supp}{Supp}

\newenvironment{proof}{\paragraph{Proof}}{}

\title{Counting Number Fields in Fibers}

\author{Yuri Bilu\footnote{Institut de Mathématiques de Bordeaux; \textsf{yuri@math.u-bordeaux.fr}}\\
\stepcounter{footnote}\stepcounter{footnote}
(with an appendix by Jean Gillibert\footnote{Institut de Mathématiques de Toulouse}\ )}

1

\makeatletter

\renewcommand*\l@section[2]{%
  \ifnum \c@tocdepth >\z@
    \addpenalty\@secpenalty
    \addvspace{0.2em \@plus\p@}%
    \setlength\@tempdima{1.5em}%
    \begingroup
      \parindent \z@ \rightskip \@pnumwidth
      \parfillskip -\@pnumwidth
      \leavevmode \bfseries
      \advance\leftskip\@tempdima
      \hskip -\leftskip
      #1\nobreak\hfil \nobreak\hb@xt@\@pnumwidth{\hss #2}\par
    \endgroup
  \fi}
  
 \makeatother

\begin{document}

\hfuzz 6pt

\maketitle

\begin{abstract}
Let~$X$ be a projective curve over~$\Q$ and ${t\in \Q(X)}$ a non-constant rational function of degree ${n\ge 2}$.  For every ${\tau\in \Z}$ pick ${P_\tau\in X(\bar\Q)}$ such that ${t(P_\tau)=\tau}$.  Dvornicich and Zannier proved that, for  large~$N$,  the field $\Q(P_1, \ldots, P_N)$ is of degree at least $e^{cN/\log N}$ over~$\Q$, where ${c>0}$ depends only on~$X$ and~$t$. In this paper we extend this result, replacing~$\Q$ by an arbitrary number field. 
\end{abstract}

{\footnotesize 

\tableofcontents

}

\section{Introduction}

\textit{Everywhere in this paper ``curve'' means ``smooth geometrically irreducible projective algebraic curve''. }

\bigskip

Let~$X$ be a  curve over~$\Q$  and ${t\in \Q(X)}$ a non-constant rational function of degree ${n\ge 2}$. 
According to the Hilbert Irreducibility Theorem, for infinitely many (in fact, ``overwhelmingly many'') ${\tau\in \Z}$  the fiber ${t^{-1}(\tau)\subset X(\bar \Q)}$ is $\Q$-irreducible; that is, the Galois group ${G_\Q=G_{\bar\Q/\Q}}$ acts on $t^{-1}(\tau)$ transitively. This can also be re-phrased as follows: for every ${\tau \in \Z}$ pick ${P_\tau\in t^{-1}(\tau)}$; then for infinitely many ${\tau\in \Z}$ we have ${[\Q(P_\tau):\Q]=n}$.  See Subsection~\ref{shilb} for a precise statement.

Hilbert's Irreducibility Theorem, however, does not answer the following natural question: among the field $\Q(P_\tau)$, are there ``many'' distinct? This question is addressed in the article of Dvornicich and Zannier~\cite{DZ94}, where the following theorem is proved (see \cite[Theorem~2(a)]{DZ94}). 

\begin{theorem} 
\label{tdvz}
In the above set-up, there exist real numbers ${c>0}$, depending on~$n$ and on the genus ${\bfg=\bfg(X)}$ and ${N_0> 1}$, depending on~$X$ and~$t$, 
such that, for every  integer ${N\ge N_0}$ the number field ${\Q(P_1,\ldots, P_N)}$ is of degree at least $e^{cN/\log N}$ over~$\Q$.
\end{theorem}


One may note that the statement holds true independently of the choice of the points $P_\tau$. 

An immediate consequence is the following result.

\begin{corollary}
\label{cdvz}
In the above set-up, there exist real numbers ${c=c(\bfg,n)>0}$ and ${N_0=N_0(X,t)>1}$  such that, for every integer ${N\ge N_0}$,  among the number fields $\Q(P_1), \ldots, \Q(P_N)$ there are at least $cN/\log N$ distinct. 
\end{corollary}

Theorem~\ref{tdvz} is best possible, as obvious examples show.  Say, if~$X$ is (the projectivization of) the plane curve ${t=u^2}$ and~$t$ is  the coordinate function, then  the field 
$$
\Q(P_1,\ldots, P_N)=\Q(\sqrt1,\sqrt2, \ldots,\sqrt N)=\Q(\sqrt p: p\le N)
$$
is of degree ${2^{\pi(N)}\le e^{cN/\log N}}$. 
On the contrary, Corollary~\ref{cdvz} is not  best possible and was recently refined in~\cite{BL16}. See the introduction of~\cite{BL16} for a brief discussion. 

The purpose of the present article is extending Theorem~\ref{tdvz} from the base field~$\Q$ to an arbitrary number field. Such an extension is required for certain applications; see, for instance,~\cite{BG16}. Our principal result is the following theorem. 


\begin{theorem}
\label{tmain}
Let~$K$ be a number field of degree~$d$ over~$\Q$. Further, let~$X$ be a  curve over~$K$ of genus~$\bfg$ and ${t\in K(X)}$ a non-constant rational function of degree ${n\ge 2}$. 
There exist  real numbers ${c=c(K,\bfg,n)>0}$ and ${B_0=B_0(K,X,t)>1}$ such that the following holds. Pick ${P_\tau \in t^{-1}(\tau)}$ for every ${\tau \in \OO_K}$. Then for every ${B\ge B_0}$  the number field 
$$
K(P_\tau:\ \tau\in \OO_K,\ \Height(\tau)\le B)
$$ 
is of degree at least $e^{cB^d/\log B}$ over~$K$. 
\end{theorem}

Here $\Height(\cdot)$ is the standard absolute height function, see Section~\ref{snota}. 

\bigskip

Again, we have the following immediate consequence. 

\begin{corollary}
\label{cmain}
In the set-up of Theorem~\ref{tmain} there exist   ${c=c(K,\bfg,n)>0}$ and ${B_0=B_0(K,X,t)>1}$ such that the following holds. Pick ${P_\tau \in t^{-1}(\tau)}$ for every ${\tau \in \OO_K}$. Then for every ${B\ge B_0}$, among the number fields 
$$
K(P_\tau) \qquad (\tau\in \OO_K,\quad \Height(\tau)\le B)
$$ 
there are  at least $cB^d/\log B$ distinct fields.  
\end{corollary}

In Sections~\ref{snota}--\ref{sram} we obtain various auxiliary results.   Theorem~\ref{tmain} is proved in Section~\ref{sardz}. In the Appendix, Jean Gillibert suggests a more canonical approach to the results of Section~\ref{sram}.

\paragraph{Acknowledgments}
A substantial part of this article was written during my stay the Max-Planck-Institut für Mathematik in Bonn. I thank this institute for the financial support and stimulating working conditions.

This article belongs to a joint project with Jean Gillibert. I  thank him for allowing me to publish this part of this project as a separate article, for adding a beautiful appendix, and for many stimulating discussions. 

\section{Preliminaries}
\label{snota}

\subsection{Number Fields, Heights and Sizes}

Given a number field~$K$, we denote by $M_K$, $M_K^\infty$ and $M_K^0$  the sets of all, infinite and finite places of~$K$, respectively. To every place ${v\in M_K}$ we associate the absolute value ${|\cdot|_v}$ normalized to extend a standard absolute value on~$\Q$: if ${v\mid \infty}$ then ${|2017|_v=2017}$, and if ${v\mid p<\infty}$ then ${|p|_v=p^{-1}}$. We denote by~$K_v$  the $v$-adic completion of the field~$K$. 

We also associate, to every finite place ${v\in M_K^0}$, the additive valuation $v(\cdot)$ normalized so that ${v(K^\times)=\Z}$; equivalently,   ${v(\norm v)=[K_v:\Q_v]}$, where $\norm v$ denotes the absolute norm of the prime ideal of~$v$.  By convention, we set ${\norm v=1}$  for ${v\in M_K^\infty}$.

We denote by $\Height(\alpha)$ the multiplicative absolute height of an algebraic number~$\alpha$: if~$K$ is a number field containing~$\alpha$ then
$$
\Height(\alpha) = \prod_{v\in M_K}\max\{1,|\alpha|_v \}^{[K_v:\Q_v]/[K:\Q]}. 
$$
We will also widely use the ``old-fashioned'' notion of the \textsl{size} $\size\alpha$ of an algebraic number~$\alpha$. If~$\alpha$ belongs to a number field~$K$  then we set
$$
\size\alpha=\max\{|\alpha|_v: v\in M_K^\infty\}. 
$$
Together with this ``upper size'' one can also define the ``lower size'' 
$$
\lsize\alpha=\min\{|\alpha|_v: v\in M_K^\infty\}. 
$$
We have clearly
\begin{align}
&\lsize\alpha=\size{\alpha^{-1}}^{-1}&&(\alpha \in  K\smallsetminus\{0\}),\nonumber\\
\label{esizeheight}
&\lsize\alpha\le \Height(\alpha) \le \size\alpha&&(\alpha \in  \OO_K\smallsetminus\{0\}),\\
&\size{\alpha+\beta}\le \size\alpha+\size\beta, \qquad \size{\alpha\beta}\le \size\alpha\cdot\size\beta && (\alpha, \beta\in K).\nonumber
\end{align}

\subsection{Algebraic Curves}
\label{sscurves}
Recall that, unless the contrary is stated explicitly, everywhere in this note ``curve'' means ``smooth geometrically irreducible projective algebraic curve''.

Let~$X$ be a curve over a field~$K$ of characteristic~$0$. We fix an algebraic closure~$\bar K$ of~$K$ and we denote by $G_K$ the absolute Galois group of~$K$, that is, the Galois group of $\bar K/K$. 

We call a \textsl{$K$-place} of the field $K(X)$ a class of non-trivial valuations on $K(X)$ whose restriction to~$K$ is trivial. 
We have a bijective correspondence given by 
$$
P\leftrightarrow \text{the class of $v_P(\cdot)$}
$$  
between the set $X(\bar K)$ of $\bar K$-points of~$X$, and the set of $\bar K$-places of $\bar K(X)$. (Here $v_P(\cdot)$ is, of course, the order of vanishing at~$P$.) In the sequel we tacitly identify the two sets and may speak on a $\bar K$-point~$P$ on~$X$ as a $\bar K$-place~$P$ of $\bar K(X)$, and vice versa. 

More generally, there is a bijection between the sets of places of $K(X)$ and the set $X(\bar K)/G_K$ of $G_K$-orbits of $\bar K$-points, and we again identify the two sets.  If ${P\in X(\bar K)}$ then the residue field of the place $P^{G_K}$ is isomorphic to $K(P)$.

Let ${t\in K(X)}$ be a $K$-rational function.  
For any point~$P$ of~$X$ there is a well-defined ``value'' ${t(P)\in K(P)\cup\{\infty\}}$. It may be defined as the only element~$\tau$ of  ${K(P)\cup\{\infty\}}$ such that ${v_P(t-\tau)>0}$.  

\textsl{Here and everywhere else throughout the article we use the standard convention ${t-\infty =t^{-1}}$. }

Now let ${t,u\in K(X)}$ be non-constant $K$-rational functions. We say that a point ${P\in X(\bar K)}$ is \textsl{$(t,u)$-regular} if the following two conditions are satisfied. 

\renewcommand{\theenumi}{R\arabic{enumi}}

\begin{enumerate}
\item
\label{idist}
For any point ${P'\ne P}$  we have ${(t(P'),u(P'))\ne (t(P),u(P))}$. 

\item
\label{iord}
We have 
$$
\min \{v_P(t-t(P)),v_P(u-u(P))\}=1.
$$
\end{enumerate}
If one of these conditions is not satisfied then we call~$P$ a $(t,u)$-singular point. 

For instance, let~$z$ be the rational function on~$\P^1$ such that ${z(P)=x/y}$ for ${P=(x:y)}$ with ${y\ne 0}$, and ${z(1:0)=\infty}$. Then for ${t=z^2}$ and ${u=z^3}$  the points ${P_0=(0:1)}$ and ${P_\infty =(1:0)}$ are $(t,u)$ singular, because 
$$
\min\{v_{P_0}(t), v_{P_0}(u)\}= \min\{v_{P_\infty}(t^{-1}), v_{P_\infty}(u^{-1})\}=2,
$$
and all the other points on~$\P^1$ are $(t,u)$-regular. And for ${t=z(z-1)}$ and ${u=z^2(z-1)}$ the points~$P_0$, ${P_1=(1:1)}$ and~$P_\infty$ are $(t,u)$-singular, because 
$$
t(P_0)=t(P_1), \quad u(P_0)=u(P_1), \quad \min\{v_{P_\infty}(t^{-1}), v_{P_\infty}(u^{-1})\}=2,
$$
and all the other points  are $(t,u)$-regular.

\renewcommand{\theenumi}{\arabic{enumi}}

The following properties will be used in the article  without special reference.

\begin{proposition}
\label{pnsing}
Let~$X$ be a curve over a field~$K$ of characteristic~$0$ and ${t,u\in K(X)}$ non-constant rational function on~$X$. 

\begin{enumerate}

\item
\label{ifielfun}
If there is at least one $(t,u)$-regular $\bar K$-point then  ${K(X)=K(t,u)}$.

\item[]
From now on assume that ${K(X)=K(t,u)}$.

\item
\label{ifieldef}
If ${P\in X(\bar K)}$ is $(t,u)$-regular then  ${K(P)=K(t(P),u(P))}$

\item
\label{insig}
Assume that ${K(X)=K(t,u)}$ and let ${F(T,U)\in \bar K[T,U]}$ be a $\bar K$-irreducible polynomial satisfying ${F(t,u)=0}$. Then for a point ${P\in X(\bar K)}$ with 
$$
t(P)=\tau\ne \infty, \qquad u(P)=\omega\ne\infty 
$$
the following properties are equivalent.

\begin{itemize}
\item
The point~$P$ is $(t,u)$-singular.

\item
We have ${F'_T(\tau,\omega)=F'_U(\tau,\omega)=0}$. 

\end{itemize}

\item
\label{ifin}
There exist at most finitely many $(t,u)$-singular points ${P\in X(\bar K)}$. 
\end{enumerate}
\end{proposition}

All this is well-known, but we include the proof for completeness. 

\paragraph{Proof of Proposition~\ref{pnsing}}
Since~$X$ is geometrically irreducible, the constant subfield of $K(X)$ is~$K$. Hence in item~\ref{ifielfun} it suffices to show that ${\bar K(X)=\bar K(t,u)}$ if there is at least one $(u,t)$-non-singular point ${P\in X(\bar K)}$.

Thus, assume that ${\bar K(t,u)}$ is a proper subfield of ${\bar K(X)}$, of degree ${m>1}$. Then there are two possibilities for our  point~$P$: 
\begin{itemize}
\item
the place~$P$ of $\bar K(X)$ is totally ramified over ${\bar K(t,u)}$, in which case both ${v_P(t-t(P))}$ and ${v_P(u-u(P))}$ are divisible by~$m$, contradicting condition~\ref{iord};

\item
the restriction of~$P$ to the field $\bar K(t,u)$ coincides with the restriction of a different place~$P'$, in which case 
${t(P)=t(P')}$ and ${u(P)=u(P')}$, contradicting condition~\ref{idist}.

\end{itemize}
This proves item~\ref{ifielfun}.

{\sloppy

Now assume that ${K(t(P),u(P))}$ is a proper subfield of $K(P)$. Pick ${\sigma \in G_{K(t(P),u(P))}\smallsetminus G_{K(P)}}$. Then we have ${P\ne P^\sigma}$, but ${t(P)=t(P^\sigma)}$ and ${u(P)=u(P^\sigma)}$, contradicting condition~\ref{idist}. This proves item~\ref{ifieldef}.

}

Item~\ref{insig} can be found in any ``old-fashioned'' course of the theory of plane algebraic curves;  for instance,  see 
Theorem~5.8 in \cite[Chapter~IV]{Wa50}. Finally, item~\ref{ifin} follows from item~\ref{insig}. \qed

\section{Lemmas on Ideals}
\label{ssideals}

In this section~$K$ is a number field of degree ${d=[K:\Q]}$ and ${\norm=\norm_{K/\Q}}$ is the $K/\Q$-norm. The index $K/\Q$ will be omitted when this does not cause confusion.

\begin{lemma}[a ``reduced'' generator of a principal ideal]
\label{lconv}
There exists a positive number~$\kappa$ (depending only on~$K$) such that the following holds. Let~$\gera$ be a principal 
fractional ideal of~$K$. Then~$\gera$ has a generator~$\alpha$ satisfying
\begin{equation}
\label{econv}
\kappa^{-1}(\norm\gera)^{1/d}\le \lsize\alpha\le \size\alpha \le \kappa(\norm\gera)^{1/d}.
\end{equation} 
\end{lemma}

\begin{proof}
This property is well-known and widely used in the Diophantine Analysis, but we include a quick proof for the reader's convenience. To simplify the notation, we denote by~$S$  the set of infinite places $M_K^\infty$.  Let ${\Log:K^\times \to \R^S}$ be the  ``logarithmic map'' ${\alpha\mapsto (\log|\alpha|_v)_{v\in S}}$ and let  ${\varsigma:\R^S\to\R}$ be the linear functional ${\bfx=(x_v)_{v\in S}\mapsto\sum_{v\in S}(d_v/d)x_v}$, where ${d_v=[K_v:\Q_v]}$ is the local degree of~$v$ (equal to~$1$ or~$2$ depending on whether~$v$ is real or complex). Then for ${\alpha\in K^\times}$ we have ${\varsigma(\Log\alpha)=d^{-1}\log|\norm\alpha|}$. In addition to this, we denote by~$\one$ the vector  ${(1)_{v\in S}\in \R^S}$ (every component is~$1$); note that ${\varsigma(\one) =1}$. 

According to the Dirichlet Unit Theorem, the image $\Log\OO_K^\times$ of the unit group forms a lattice in $\ker\varsigma$. Hence 
there exists ${\lambda>0}$ such that for any ${\bfx\in \ker \varsigma}$ there exits ${\bfx'\in \ker\varsigma}$ satisfying ${\bfx\equiv \bfx'\bmod \Log\OO_K^\times}$ and ${\|\bfx'\|_\infty\le\lambda}$, where ${\|\cdot\|_\infty}$ stands for the sup-norm. More generally, for an arbitrary ${\bfx\in \R^S}$, we find, by applying the previous sentence to the vector ${\bfx-\nu(\bfx)\one\in \ker\varsigma}$, a vector ${\bfx'\in \R^S}$ satisfying  ${\bfx\equiv \bfx'\bmod \Log\OO_K^\times}$ and ${\|\bfx'-\nu(\bfx)\one\|_\infty\le\lambda}$. In particular, for  ${\beta \in K^\times}$ we can find ${\alpha \in K^\times}$ such that ${\beta/\alpha\in \OO_K^\times}$ and 
$$
e^{-\lambda}|\norm\beta|^{1/d}\le |\alpha|_v\le e^{\lambda}|\norm\beta|^{1/d}
$$
for all ${v\in S}$. Taking~$\beta$ as a generator of the principal ideal~$\gera$, we find thereby another generator~$\alpha$ satisfying~\eqref{econv} with ${\kappa=e^\lambda}$. \qed
\end{proof}

\begin{lemma}[a ``reduced'' $\Z$-basis of an ideal]
\label{lbas}
There exists a positive number~$\kappa$ (depending only on~$K$) such that the following holds. Let~$\gera$ be a  
fractional ideal of~$K$. Then~$\gera$ has a $\Z$-basis ${\alpha_1, \ldots, \alpha_d}$ satisfying
\begin{equation}
\label{ebas}
\kappa^{-1}(\norm\gera)^{1/d}\le \lsize{\alpha_i}\le \size{\alpha_i} \le \kappa(\norm\gera)^{1/d}. \qquad (i=1, \ldots, d).
\end{equation} 
\end{lemma}

\begin{proof}
There exists a real number~$\lambda$, depending only on~$K$, such that the following holds: every ideal class of~$K$ has an ideal~$\gerb$ satisfying 
${\lambda^{-d}\le \norm\gerb\le \lambda^d}$
and having a $\Z$-basis ${\beta_1,\ldots,\beta_d}$ such that 
\begin{equation*}
\lambda^{-1}\le \lsize{\beta_i}\le \size{\beta_i} \le \lambda. \qquad (i=1, \ldots, d).
\end{equation*} 
In particular, such~$\gerb$ can be found in the class of our ideal~$\gera$. Lemma~\ref{lconv} implies that the principal ideal ${\gera\gerb^{-1}}$ has a generator~$\gamma$ satisfying 
\begin{equation*}
(\kappa')^{-1}(\norm(\gera\gerb^{-1}))^{1/d}\le \lsize\gamma\le \size\gamma \le \kappa'(\norm(\gera\gerb^{-1}))^{1/d},
\end{equation*} 
where~$\kappa'$ depends only on~$K$. Setting ${\alpha_i=\beta_i\gamma}$, we obtain a $\Z$-basis ${\alpha_1, \ldots, \alpha_d}$ of~$\gera$ satisfying~\eqref{ebas} with ${\kappa=\kappa'\lambda^2}$. \qed
\end{proof}

\bigskip

Given a $K$-prime~$\gerp$, an element ${\pi\in K}$ is called $\gerp$-\textsl{primitive} if ${v_\gerp(\pi)=1}$, where~$v_\gerp$ is the place associated to~$\gerp$. Since a $\Z$-basis of~$\gerp$ has at least one primitive element, Lemma~\ref{lbas} has the following consequence. 

{\sloppy

\begin{corollary}[a ``reduced'' primitive element]
\label{cprim}
There exists a positive number~$\kappa$ (depending only on~$K$) such that the following holds.
For every $K$-prime~$\gerp$  there exists a $\gerp$-primitive ${\pi\in \OO_K}$ satisfying 
$$
\kappa^{-1}(\norm\gerp)^{1/d}\le \lsize\pi\le \size\pi\le \kappa(\norm\gerp)^{1/d} .
$$
\end{corollary}

}

Another application of Lemma~\ref{lbas} is locating ``reduced'' elements in residue classes.  

{\sloppy

\begin{corollary}[a ``reduced'' element in a residue class]
\label{cmod}
There exists a real  ${\kappa\ge 1}$ (depending only on~$K$) such that the following holds.
Let~$\gera$ be a non-zero ideal of $\OO_K$. Then for every  ${\alpha \in \OO_K}$ there exists ${\beta\in \OO_K}$ such that ${\alpha\equiv\beta\bmod \gera}$ and ${\size\beta\le \kappa(\norm\gera)^{1/d}}$.
\end{corollary}

}

\begin{proof}
It is a standard lattice argument. We identify~$K$ with its image (under the diagonal embedding) in ${V=\R^{s_1}\times\bbC^{s_2}}$, where~$s_1$ and $2s_2$ are the numbers of real and complex embeddings of~$K$. Then~$\gera$ becomes a lattice in~$V$, and every element of~$V$ is congruent modulo~$\gera$ to an element of its fundamental domain. According to Lemma~\ref{lbas}, the lattice~$\gera$ has a basis ${\alpha_1,\ldots,\alpha_d}$ satisfying ${\size{\alpha_i}\le \kappa'(\norm\gera)^{1/d}}$, where~$\kappa'$ depends only on~$K$. Every element of the fundamental domain spanned by this basis  is of size at most ${d\kappa'(\norm\gera)^{1/d}}$. This proves the corollary with ${\kappa=d\kappa'}$. \qed
\end{proof}

\begin{remark}
Lower estimates for the ``lower size'' obtained in~\eqref{econv},~\eqref{ebas} etc. will not be used elsewhere in this article; we include them only for completeness. Moreover, the proof of Corollary~\ref{cmod} can be easily modified to obtain a lower estimate as well. We do not do it because we will not need this lower estimate.

\end{remark}

\section{Thin Subsets and Hilbert's Irreducibility Theorem}
\label{shilb}

In this section we recall basic definitions and facts about thin sets, and state Hilbert's Irreducibility Theorem.

Let~$K$ be a field of characteristic~$0$.  We call  ${\mho \subset K}$  a \textsl{basic thin subset} of~$K$ if there exists a (smooth geometrically irreducible) curve~$Y$ defined over~$K$ and a non-constant rational function ${u\in K(X)}$ of degree at least~$2$  such that ${\mho \subset u(Y(K))}$. A \textsl{thin subset} of~$K$ is a union of finitely many basic thin subsets. Thin subsets form an ideal in the algebra of subsets of~$K$. Serre in \cite[Section~9.1]{Se97} gives a differently looking, but equivalent definition of thin sets.

Any finite set is thin, and if~$K$ is algebraically closed then any  subset of~$K$ is thin. 

\begin{remark}
If~$L$ is an extension of~$K$ then any thin subset of~$K$ is also thin as a subset of~$L$. The converse is true when~$L$ is finitely generated over~$K$ \cite[Proposition~2.1]{BL05} but not in general; for instance, any number field~$K$ is a thin subset of its algebraic closure~$\bar K$ but is not a thin subset of itself by the Hilbert Irreducibility Theorem quoted below. 
\end{remark}

Using elementary Galois theory one easily proves the following (see \cite[Section~9.2]{Se97})

\begin{proposition}
\label{pgalo}
Let~$X$ be a curve over~$K$ and ${t\in K(X)}$ a non-constant rational function. Then the set of ${\tau\in K}$ such that the fiber ${t^{-1}(\tau)}$ is reducible over~$K$ is thin. 
\end{proposition}

Hilbert's Irreducibility Theorem asserts that when~$K$ is a number field then its ring of integers $\OO_K$ is not a thin subset of~$K$. In fact, one has the following counting result (see~\cite{Se97}, Theorem on page~134). 

\begin{theorem}
\label{tthin}
Let~$K$ be a number field of degree~$d$ over~$\Q$ and ${\mho\subset \OO_K}$ a thin subset of~$K$. Then for  ${B\ge 1}$ the set~$\mho$ has  $O(B^{d/2})$ elements~$\alpha$ with ${\size\alpha\le B}$, the implicit constant depending  on~$K$ and on~$\mho$. 
\end{theorem}

Combining this with Proposition~\ref{pgalo}, we obtain the following ``quantitative'' version of  Hilbert's Irreducibility Theorem.

\begin{corollary}
\label{chilb}
Let~$K$ be a number field of degree~$d$ over the rationals,~$X$ a curve over~$K$ and ${t\in K(X)}$ a non-constant rational function.  Then for  ${B\ge 1}$ there exist at most $O(B^{d/2})$ elements ${\tau \in \OO_K}$ with ${\size\alpha\le B}$ such that  the fiber ${t^{-1}(\tau)}$ is reducible over~$K$. Here the implicit constant depends only on~$K$,~$X$ and~$t$. 
\end{corollary}

\section{Local Behavior of Functions on a Curve}

In this section we compare the behavior of two distinct functions in a neighborhood of a point on an algebraic curve. Our main tool will be the Puiseux expansion. 

Unless the contrary is stated explicitly, in this section~$X$ is a (smooth projective) curve over a number field~$K$ and ${t,u\in K(X)}$ non-constant $K$-rational functions. Further, let ${A\in X(K)}$ be a $K$-rational point. For ${v \in M_K}$ we want to compare~$t$ and~$u$ in a $v$-adic neighborhood of~$A$. 

\begin{theorem}
\label{tpuis}
Assume that~$A$ is a $(t,u)$-regular point of~$X$ (as defined in Subsection~\ref{sscurves}). 
There exists a finite subset ${S\subset M_K}$ (depending on~$X$,~$t$,~$u$ and~$A$) such that for ${v\in M_K\smallsetminus S}$ the following holds. Assume that ${P\in X(K_v)}$ satisfies 
$$
|t(P)-t(A)|_v<1, \qquad |u(P)-u(A)|_v<1. 
$$
Then 
$$
|t(P)-t(A)|_v^{1/v_A(t)}= |u(P)-u(A)|_v^{1/v_A(u)}
$$ 
\end{theorem}

\begin{remark}
The $(t,u)$-regularity hypothesis can be relaxed: in fact, it suffices to assume that our point~$A$ satisfies only condition~\ref{idist} in the definition of $(t,u)$-regularity in Subsection~\ref{sscurves}, while
condition~\ref{iord} may be be suppressed. But the present form of the theorem is sufficient for us, and assuming~\ref{iord} slightly simplifies the proof. 
\end{remark}

\subsection{Puiseux Expansion}

Let us briefly recall the notion of the Puiseux expansion. Let~$K$ be a field,~$X$ a smooth projective curve over~$K$ and ${A\in X(K)}$ a $K$-rational point. Further, let ${t\in K(X)}$ be a non-constant $K$-rational function with ${v_A(t)=1}$. Then one can realize the completion of $K(X)$ with respect to the valuation $v_A(\cdot)$ as  the field of formal power series $K((t))$. In particular, we view $K(X)$ as a subfield of $K((t))$, the function ${t\in K(X)}$ being identified with ${t\in K((t))}$. 

If ${u\in K(X)}$ is another $K$-rational function on~$X$,  then its image in $K((t))$ is a certain power series ${\sum_{k=\nu}^\infty a_k t^k}$ with ${\nu=v_A(u)}$ and ${a_\nu\ne 0}$. We call this series the \textsl{Puiseux expansion of~$u$ at~$A$ in~$t$}.

Now assume that~$K$ is a number field. Then the coefficients~$a_k$ of the Puiseux expansion satisfy the classical \textsl{Eisenstein Theorem}, which says, informally, that for all ${v\in M_K}$ the $v$-adic norm of the coefficients grows at most exponentially in~$k$, and for all but finitely many~$v$ they are bounded by~$1$. In symbols:    for every ${v\in M_K}$ there exists ${C_v\ge 1}$, such that ${C_v=1}$ for almost all~$v$, and 
$$
|a_k|_v \le C_v^{k-\nu+1}\qquad (k\ge \nu, \ v\in M_K). 
$$
We will only need the following weaker result. 

\begin{proposition}[Eisenstein]
\label{peis}
There exists a finite set ${S\subset M_K}$ (containing all the infinite places) such that for every  ${v\in M_K\smallsetminus S}$ the coefficients~$a_k$ are $v$-adic integers. 
\end{proposition}

We want to show now that for all but finitely many~$v$ the Puiseux expansion indeed expresses~$u$ in terms of~$t$ in a suitable $v$-adic ``neighborhood'' of the point~$A$. 

\begin{proposition}
\label{ppuis}
In the set-up of this subsection, assume that~$A$ is  a  $(t,u)$-regular point of~$X$. Then 
there exists a finite set ${S\subset M_K}$ (which contains all the infinite places and might be different from the set~$S$ of Proposition~\ref{peis}) such that for every  ${v\in M_K\smallsetminus S}$ the coefficients $a_k$ are $v$-adic integers and the following holds. Assume that ${P\in X(K_v)}$ satisfies 
${|t(P)|_v<1}$ and ${|u(P)-u(A)|_v<1}$. 
Then the series ${\sum_{k=\nu}^\infty a_k t(P)^k}$ converges $v$-adically to $u(P)$. 
\end{proposition}

\paragraph{Proof}
We may assume without loss of generality that\footnote{This is obvious if ${u(A)\ne \infty}$ (just replace~$u$ by ${u-u(A)}$~), but requires some explanation in the case when~$A$ is a pole of~$u$. In this latter case the set~$S$ should be extended to make the leading coefficient of the Puiseux expansion for~$u$ an $S$-unit. Then for ${v\notin S}$ the coefficients of the series for~$u$ are $v$-adic integers if and only if the same holds for $1/u$. And if, in addition to this,  ${|t(P)|_v<1}$ and the series for $1/u$ converges $v$-adically at $t(P)$ to $1/u(P)$, then the series for~$u$ converges at $t(P)$ to $u(P)$.} 
${u(A)=0}$. Then 
$$
\nu=v_A(u)\ge 1.
$$
Let  ${F(T,U)\in K[T,U]}$ be a $K$-irreducible polynomial such that ${F(t,u)=0}$; in particular,
\begin{equation}
\label{ezero}
F(0,0)=F(t(A),u(A))=0.
\end{equation}
Further, let ${A_1=A,A_2,\ldots, A_s\in X(\bar K)}$ be all points which are zeros of~$t$ and which are \textsl{not} poles of~$u$. Then ${u(A_1),\ldots, u(A_s)}$ are the roots of the polynomial $F(0,U)$, of multiplicities ${v_{A_1}(t), \ldots, v_{A_s}(t)}$, and it has no other roots. Since ${A=A_1}$ is a $(t,u)$-regular point,  we have ${u(A)\ne u(A_i)}$ for ${i=2,\ldots, s}$. In particular, 
${0=u(A)}$ is a root of $P(0,U)$ of multiplicity ${v_A(t)=1}$. In other words, 
${F'_U(0,0)\ne 0}$, and we normalize the polynomial~$F$ to have 
\begin{equation}
\label{enzero}
F'_U(0,0)=1.
\end{equation}

Now let~$S$ be a a finite subset of~$M_K$ like in Proposition~\ref{peis}. Enlarging it, we may assume that for ${v\in M_K\smallsetminus S}$ 
all the coefficients of the polynomial~$F$ are $v$-adic integers.

Now fix ${v\in M_K\smallsetminus S}$ and let ${P\in X(K_v)}$ be such that 
$$
|t(P)|_v<1, \qquad |u(P)|_v<1.
$$
Set ${\tau=t(P)}$. Since ${|\tau|_v<1}$ and the coefficients of the polynomial~$F$ are $v$-adic integers,~\eqref{ezero} and~\eqref{enzero} imply that 
\begin{equation}
\label{ehensel}
|F(\tau,0)|_v<1, \qquad |F'_U(\tau,0)|_v=1. 
\end{equation}
Furthermore, since  ${|\tau|_v<1}$ and the coefficients~$a_k$ are $v$-adic integers,  the series ${\sum_{k=\nu}^\infty a_k \tau^k}$ converges in $K_v$ to a sum that we denote by~$\omega$. Since ${\nu\ge 1}$, we have ${|\omega|_v<1}$, and since  ${F(t,\sum_{k=\nu}^\infty a_k t^k)=0}$, we have ${F(\tau,\omega)=0}$. On the other hand, 
${F(\tau, u(P))= F(t(P),u(P))=0}$
as well. 

Thus, both~$\omega$ and~$u(P)$ are roots of the polynomial $F(\tau,U)$. 
However,  Hensel's lemma implies that, in view of~\eqref{ehensel}, 
this polynomial may have only one root of $v$-adic norm strictly smaller than~$1$. Hence ${u(P)=\omega}$, proving the proposition. \qed

\subsection{Proof of Theorem~\ref{tpuis}} 
It is an easy consequence of Proposition~\ref{ppuis}. We may assume that 
$$
t(A)=u(A)=0
$$
and ${v_A(t)=1}$. Then ${v_A(u)=\nu\ge 1}$. 
Let ${\sum_{k=\nu}^\infty a_k t^k}$  be the Puiseux expansion of~$u$ at~$A$, and let~$S$ be as in Proposition~\ref{ppuis}. Enlarging~$S$, we may assume that ${|a_\nu|_v=1}$ for ${v\in M_K\smallsetminus S}$. 

{\sloppy

Now fix ${v\in M_K\smallsetminus S}$. Then for any ${P\in X(K_v)}$ satisfying ${|t(P)|_v<1}$ and ${|u(P)|_v<1}$ we have ${u(P)=\sum_{k=\nu}^\infty a_k t(P)^k}$. Since ${|a_\nu|_v=1}$, we obtain ${|u(P)|_v=|t(P)|_v^\nu}$, whence the result. \qed

}

\section{Polynomials over Complete Fields}

In this section we collect results, mainly well-known, on polynomials over complete field. 

\subsection{Roots of Polynomials and Power Series}

Let~$K$ be a field of characteristic~$0$, complete with respect to a non-archimedean absolute value ${|\cdot|}$. 
Let ${f(T)\in K [T]}$ be a polynomials having a root ${\alpha \in K}$ of multiplicity~$e$. It is well-known that if ${g(T)\in K[T]}$ is another polynomial of the same degree, ``sufficiently close'' to~$f$, then ${g(T)}$ has~$e$ roots in~$\bar K$ ``close'' to~$\alpha$; see Proposition~7.1 in \cite[Chapter~XII]{La02} as an example of such a statement.

We need a precise form of this statement. For a polynomial 
$$
f(T)=a_nT^n+\cdots +a_0\in K[T]
$$ 
we use notation 
$$
|f|=\max\{|a_0|, \ldots, |a_n|\}. 
$$
We extend the absolute value from~$K$ to its algebraic closure~$\bar K$. 

\begin{theorem}
\label{troots}
Let ${f(T)=a_nT^n+\cdots +a_0\in K[T]}$ be a polynomial of degree~$n$ having pairwise distinct roots ${\alpha_1, \ldots, \alpha_s\in K}$ of multiplicities ${e_1, \ldots, e_s}$, respectively, and no other roots in~$\bar K$ (so that ${e_1+\cdots+e_s=n)}$. Assume that
\begin{align}
\label{erootsint}
|f|=|a_n|&=1;\\
\label{erootsapart}
|\alpha_i-\alpha_j|&=1 &&(1\le i<j\le s);\\
\label{ederiv}
\left|\frac{f^{(e_i)}(\alpha_i)}{e_i!}\right|&=1&&(1\le i\le s).
\end{align}
Let ${g(T)\in K[T]}$ be another polynomial of degree~$n$ satisfying ${|f-g|<1}$. Then the set of roots of~$g$ in~$\bar K$ splits into disjoint sets ${B_1, \ldots, B_s}$, such that every~$B_i$ has exactly~$e_i$ roots counted with multiplicities, and every ${\beta\in B_i}$ satisfies 
$$
|\beta-\alpha_i|<1, \qquad |\beta-\alpha_j|=1 \quad (i\ne j). 
$$
\end{theorem}

The proof relies on the famous theorem of Strassmann on zeros of power series in complete fields. We will use this theorem only for polynomials, but we state it for power series, in its full strength. 

Thus, let ${f(T)=\sum_{k=0}^\infty a_kT^k\in K[[T]]}$ be a formal power series over a non-archimedean complete field~$K$, whose coefficients satisfy ${|a_k|\to0}$ as ${k\to\infty}$. Then~$f$ defines an analytic function on the closed disc ${\OO=\{\alpha\in K:|\alpha|\le 1\}}$. 

\begin{theorem}[Strassmann]
\label{tstras}
Set 
$$
A=\max\{|a_k|: k=0,1,2,\ldots\}, \qquad \kappa_{\max}=\max\{k: |a_k|=A\}.
$$
Then $f(T)$ has  at most $\kappa_{\max}$ zeros ${\alpha\in K}$ with ${|\alpha|\le 1}$. 
\end{theorem}

The proof is well-known, by induction in~$\kappa_{\max}$. If ${\kappa_{\max}=0}$ then  $f(T)$ clearly does not vanish on~$\OO$. Now assume that ${\alpha\in \OO}$ is a zero of~$f$. It is easy to see that replacing $f(T)$ by ${f(\alpha+T)}$ does not alter the value of $\kappa_{\max}$; hence we may assume ${\alpha=0}$. It follows that ${a_0=0}$, and we reduce the statement for $f(T)$ to that for ${T^{-1}f(T)=a_1+a_2T+\ldots}$, reducing $\kappa_{\max}$  by~$1$. 

Here is a useful complement to Strassmann's theorem. 

\begin{corollary}
\label{cstras}
Set 
${\kappa_{\min}=\min\{k: |a_k|=A\}}$. 
Then $f(T)$ has  at most $\kappa_{\min}$ zeros ${\alpha\in K}$ with ${|\alpha|< 1}$. 
\end{corollary}

\paragraph{Proof}
We may assume that the set of zeros ${\alpha\in K}$ with ${|\alpha|< 1}$ is non-empty: otherwise there is nothing to prove. Since this set is finite by Theorem~\ref{tstras}, it has an element~$\theta$ of maximal absolute value:
$$
|\theta|=\max\{|\alpha|: f(\alpha)=0,\ |\alpha|<1\}. 
$$
Then we have to count zeros in~$\OO$ of the function $f(\theta T)$. Since ${|\theta|<1}$, the value of $\kappa_{\max}$ for the series $f(\theta T)$ does not exceed the value of $\kappa_{\min}$ for the series $f(T)$. Hence the result follows by Theorem~\ref{tstras}. \qed 

\paragraph{Proof of Theorem~\ref{troots}} We write
${ g(T)=b_nT^n+\cdots +b_0}$. 
Extending the field~$K$, we may assume that all the roots of~$g$ belong to~$K$ as well. Condition~\eqref{erootsint} implies that all roots of~$f$ are of absolute value at most~$1$. Since ${|f-g|<1}$,  we have  ${|g|=|b_n|=1}$, and the roots of~$g$ are of absolute value at most~$1$ as well.  

Let us show first of all that for every root~$\beta$ of~$g$ there exists a unique root~$\alpha_i$ of~$f$ such that ${|\beta-\alpha_i|<1}$. Indeed, uniqueness follows from~\eqref{erootsapart}, and existence from
$$
\prod_{i=1}^s|\beta-\alpha_i|^{e_i}=|f(\beta)|=|f(\beta)-g(\beta)|\le |f-g|<1. 
$$
This already defines the partition ${B_1\cup\ldots\cup B_s}$ on the set of roots of~$g$, and we only need to show that each  $B_i$ contains exactly $e_i$ roots (counted with multiplicities). Moreover, since ${n=e_1+\cdots+e_s}$  is the total number of roots of~$g$, it is sufficient to show that~$B_i$ contains \textsl{at most}~$e_i$ roots.

Thus, fix~$\alpha_i$ and omit index~$i$ in the sequel. We want to show that~$g$ has at most~$e$ roots~$\beta$ satisfying ${|\beta-\alpha|<1}$. Replacing $f(T)$ and $g(T)$ by ${f(\alpha+T)}$ and ${g(\alpha+T)}$, we may assume that ${\alpha=0}$. Thus,
${f(T)=a_eT^e+\cdots+a_nT^n}$ with ${|a_e|=1}$ by~\eqref{ederiv}.  Then for the coefficients of $g(T)$ we have ${|b_k|<1}$ for ${k<e}$ and ${|b_e|=1}$. By Corollary~\ref{cstras}, the polynomial~$g$ has at most~$e$ roots~$\beta$ satisfying ${|\beta|<1}$. This completes the proof. \qed

\bigskip

Here is a consequence for number fields. 

\begin{corollary}
\label{croots}
Let~$K$ be a number field and ${f(T)\in K[T]}$ a polynomial of degree~$n$ having a root ${\alpha\in K}$ of order~$e$. 
Then there exists a finite set ${S\subset M_K}$ (containing all the infinite places), such that for every ${v\in M_K\smallsetminus S}$ the following holds. Let ${g(T)\in K_v[T]}$ be a polynomial of degree~$n$ satisfying ${|f-g|_v<1}$. Then~$g$ has exactly~$e$ roots ${\beta\in\overline{K_v}}$ satisfying ${|\beta-\alpha|_v<1}$.
\end{corollary}

\paragraph{Proof}
If the statement holds true with~$K$ replaced by some finite extension, then it is true for~$K$ as well. Thus, extending~$K$, we may assume that all roots of~$f$ belong to~$K$. And in this case the statement is an immediate consequence of Theorem~\ref{troots}. \qed

\subsection{Ramification}

Now assume that~$K$ is a non-archimedean  local  field of characteristic~$0$; we denote by ${|\cdot|}$ its absolute value  and by ${\OO=\{\alpha \in K: |\alpha|\le 1\}}$ its local ring. 

The following property is well-known (at least when the polynomial $f(T)$ is irreducible), but we include the proof for the reader's convenience. 

\begin{proposition}
\label{pderiv}
Let ${f(T)\in \OO[T]}$ be a monic polynomial and ${\alpha \in \bar K}$ a root of~$f$ such the field $K(\alpha)$ is ramified over~$K$. Then ${|f'(\alpha)|<1}$.
\end{proposition}

\paragraph{Proof}
Note first of all that, since~$f$ is monic and has coefficients in~$\OO$, its root~$\alpha$ satisfies ${|\alpha|\le 1}$. It follows that ${|f'(\alpha)|\le 1}$. We will assume that ${|f'(\alpha)|=1}$ and obtain a contradiction.

Let~$L$ be the maximal unramified extension of~$K$ contained in $K(\alpha)$. Then there exists ${\theta\in L}$ with ${|\theta-\alpha|<1}$. Since ${f(\alpha)=0}$ and ${|f'(\alpha)|=1}$, we have ${|f(\theta)|<1}$ and ${|f'(\theta)|=1}$. 

The existence part of Hensel's Lemma implies that $f(T)$ has a root ${\alpha'\in L}$ satisfying ${|\alpha'-\theta|<1}$. The uniqueness part of Hensel's lemma implies that $f(T)$ can have at most one root in $K(\alpha)$ with this property. Hence ${\alpha=\alpha'\in L}$, a contradiction. \qed

\bigskip

We again have an immediate consequence for the number fields. 

\begin{corollary}
\label{cderiv}
Let~$K$ be a number field and ${f(T)\in K[T]}$. Let ${v\in M_K^0}$ be such that all the coefficients of~$f$ are $v$-adic integers, and the leading coefficient of~$f$ is a $v$-adic unit. Viewing~$f$ as a polynomial over~$K_v$, let ${\alpha\in \overline{K_v}}$ be its root such that the field  $K_v(\alpha)$ is ramified over~$K_v$. Then ${|f'(\alpha)|_v<1}$. 
\end{corollary}

\section{Arithmetical vs Geometric Ramification}
\label{sram}

Let~$K$ be a field of characteristic~$0$ and~$X$ a (smooth projective) algebraic curve over~$K$. Further, let ${t\in K(X)}$ be a non-constant function. For a point ${A\in X(\bar K)}$ we define the \textsl{ramification index} of~$t$ at~$A$ by 
$$
e_A=e_A(t)=v_A(t-t(A));
$$ 
We say that~$t$ is \textsl{ramified at~$A$} (and call~$A$ a \textsl{ramification point} of~$t$) if ${e_A(t)>1}$. The value ${t(A)\in \bar K\cup\{\infty\}}$ of~$t$ at a ramification point will be called a \textsl{critical value} of~$t$. (It is also often called a \textsl{branch point} of~$t$.) It is well-known that 

\begin{itemize}
\item
a non-constant rational function has at most finitely many ramification points (and critical values);

\item
if ${\bar K(t)\ne \bar K(X)}$ then~$t$ has at least~$2$ distinct critical values (a consequence of the Riemann-Hurwitz formula). 
\end{itemize}

In this section we prove three theorems linking geometric and arithmetical ramification. None of them is really new, but we did not find in the literature what we exactly need. 

{\sloppy

\begin{theorem}
\label{tram}
Let~$K$ be a number field,~$X$ a smooth projective algebraic curve over~$K$ and ${t\in K(X)}$ a non-constant $K$-rational function. Further, let 
${\alpha \in K\cup\{\infty\}}$ be a critical value of~$t$. Then there exists a finite set~$S$ of places of~$M_K$ 
such that for every ${v\in M_K\smallsetminus S}$ the following holds. Let  ${\tau\in K_v}$ be such that ${ v(\tau-\alpha)=1}$. Then there exists ${P\in X(\overline{K_v})}$ with ${t(P)=\tau}$ such that the field $K_v(P)$ is ramified over~$K_v$. 
\end{theorem}
(Recall that we normalize the discrete valuation $ v(\cdot)$ 
so that  ${ v(K^\times)=\Z}$.)

}

\bigskip

Informally, the theorem says that ``geometric ramification enforces arithmetical ramification''. 

On the contrary, if $t(P)$ is $v$-adically close to a non-critical value, then  $K_v(P)$ does not ramify over~$K_v$.

{\sloppy 

\begin{theorem}
\label{tnram}
Let~$K$ be a number field,~$X$ a smooth projective algebraic curve over~$K$ and ${t\in K(X)}$ a non-constant $K$-rational function. Further, let 
${\alpha \in K\cup\{\infty\}}$ be a \textbf{not} a critical value of~$t$. Then there exists a finite set~$S$ of places of~$M_K$ 
such that for every ${v\in M_K\smallsetminus S}$ the following holds:  for any ${P\in X(\overline{K_v})}$ with ${|t(P)-\alpha|_v<1}$ the field $K_v(P)$ is unramified over~$K_v$.   
\end{theorem}

}

This theorem is  easy (it is, basically, an application of Hensel's lemma), but quite useful. In fact, a similar statement is absolutely crucial in~\cite{BG16}.  

Theorem~\ref{tram} has a partial converse: for almost all~$v$, if  $K_v(P)$ ramifies over~$K_v$ then $t(P)$ must be $v$-adically close to a critical value. 

\begin{theorem}
\label{tclose}
Let~$K$ be a number field,~$X$ a smooth projective algebraic curve over~$K$ and ${t\in K(X)}$ a non-constant $K$-rational function. Assume 
that all the critical values belong to ${K\cup\{\infty\}}$.
Then there exist  a finite set 
${S\subset M_K}$  such that for every ${v\in M_K\smallsetminus S}$   the following holds. Let ${P\in X(\overline{K_v})}$ be such that ${t(P)\in K_v}$ and  the field  $K_v(P)$ ramifies over~$K_v$. Then there exists a unique critical value~$\alpha$ such that ${|t(P)-\alpha|_v<1}$. 
\end{theorem}

\begin{remark}
An alternative treatment of the principal results of this section, using the scheme-theoretic language, can be found in the Appendix due to Jean Gillibert. 
\end{remark}

\subsection{Proof of Theorem~\ref{tram}}
Remark first of all that we may replace~$K$ by a finite extension. Indeed, assume that the statement holds true with~$K$ replaced by a finite extension~$K'$, and let~$S'$ be the corresponding finite subset of $M_{K'}$. Then  the statement holds over~$K$ as well, if we define~$S$ as the set of places ${v\in M_K}$ which extend to some ${v'\in S'}$ or ramify in~$K'$.

We may assume  that ${\alpha=0}$. Let ${A\in X(\bar K)}$ be a ramification point of~$t$ such that ${t(A)=0}$. Extending the base field~$K$, we may assume that ${A\in X(K)}$. 
Pick a function ${u\in K(X)}$ with the following properties:
\begin{enumerate}

\item
\label{inuuone}
${v_A(u)=1}$;

\item
\label{insing}
for any point ${A'\in X(\bar K)}$, distinct from~$A$, we have 
$$
(t(A'),u(A'))\ne (t(A),u(A));
$$

\item
$u$ has no poles among the zeros of~$t$.

\end{enumerate} 
Observe that properties~\ref{inuuone} and~\ref{insing} above imply that~$A$ is a $(t,u)$-regular point of~$X$, as defined in Subsection~\ref{sscurves}. In particular, we have
${K(X)=K(t,u)}$.

Let
\begin{equation}
\label{eftu}
F(T,U)=a_n(T)U^n+\cdots+a_0(T)\in K[T,U]
\end{equation}
be such a polynomial that $F(t,U)$ is the minimal polynomial of~$u$ over $K[t]$. 
If~$\tau$ belongs to some extension of~$K$, we set ${f_\tau(U)=F(\tau,U)}$. Since~$u$ has no poles among the zeros of~$t$, the polynomial ${f_0(U)\in K[U]}$ is of degree ${n=\deg_UF}$, and ${u(A)=0}$ is its root of order ${e=e_A(t)}$. We normalize~$F$ to make~$f_0$ a monic polynomial (having leading coefficient~$1$). 

Now let ${S\supseteq M_K^\infty}$ be a finite subset of $M_K$ such that for any ${v\in M_K\smallsetminus S}$ the conclusion of Corollary~\ref{croots} holds for the polynomial~$f_0$ and its root~$0$, and the conclusion of  Theorem~\ref{tpuis} holds for the functions~$t,u$ and the point~$A$. 
Expanding the set~$S$, we may assume that for ${v\in M_K\smallsetminus S}$ the coefficients of $F(T,U)$ are $v$-adic integers. 

Now fix ${v\in M_K\smallsetminus S}$. Since the coefficients of $F(T,U)$ are $v$-adic integers, for any ${\tau\in K_v}$ with ${|\tau|_v<1}$ we have ${|f_\tau-f_0|_v<1}$. Clearly, ${\deg f_\tau\le n}$; but, since~$f_0$ is monic and ${|f_\tau-f_0|_v<1}$, we have ${\deg f_\tau=n}$. 

Corollary~\ref{croots} implies that~$f_\tau$ has a root ${\omega\in \overline{K_v}}$ with the property ${|\omega|_v<1}$. Since ${F(\tau,\omega)=0}$, there exists a point ${P\in X(\overline{K_v})}$ such that 
$$
t(P)=\tau, \qquad u(P)=\omega.
$$
For this point we have 
$$
|t(P)-t(A)|_v =|\tau|_v<1, \qquad  |u(P)-u(A)|_v=|\omega|_v<1
$$
(recall that ${t(A)=u(A)=0}$). Applying Theorem~\ref{tpuis},  
we find that 
${|\omega|_v=|\tau|_v^{1/e}}$. 

Now if ${ v(\tau)=1}$ then the field $K_v(\omega)$ must have ramification index at least~$e$ over~$K_v$. In particular, $K_v(P)$ is ramified over~$K_v$. \qed

\subsection{Proof of Theorem~\ref{tnram}}
As in the proof of Theorem~\ref{tram} we may assume ${\alpha=0}$ and we may replace our field~$K$ by a suitable finite extension. Thus, we extend~$K$ to have all points in the fiber ${t^{-1}(0)}$ defined over~$K$. Since~$0$ is not a critical value, ${t^{-1}(0)}$ consists of~$n$ distinct points (where~$n$ is the degree of~$t$). 

Now let ${u\in K[X]}$ be such that~$u$ takes pairwise distinct finite values at the points from ${t^{-1}(0)}$. This implies, in particular, that ${K(X)=K(t,u)}$. We define $F(T,U)$ as in the previous proof. Then the polynomial ${f_0(U)=F(0,U)\in K[U]}$ is of degree~$n$ and has~$n$ distinct simple roots in~$K$. In particular, ${f_0'(\alpha)\ne 0}$ for any root~$\alpha$ of~$f_0$. 

Further extending the field~$K$, we may assume that all $(t,u)$-singular points of~$X$ are $K$-rational.

Now let ${S\supseteq M_K^\infty}$ be a finite subset of $M_K$  such that  for ${v\in M_K\smallsetminus S}$ the coefficients of $F(T,U)$ are $v$-adic integers,  the leading coefficient of~$f_0$ is a $v$-adic unit, and ${|f_0'(\alpha)|_v=1}$ for any root~$\alpha$ of~$f_0$. 

Fix ${v\in M_K\smallsetminus S}$ and let ${P\in X(\overline{K_v})}$ be such that ${t(P)=\tau\in K}$ and ${|\tau|_v<1}$. We may assume~$P$ to be $(t,u)$-regular; otherwise ${K_v(P)=K_v}$ and there is nothing to prove.

As in the previous proof, the polynomial ${f_\tau(U)=F(\tau, U)}$ satisfies 
$$
|f_0-f_\tau|_v<1,
$$ 
and its leading coefficient is a $v$-adic unit. Since~$P$ is $(t,u)$-regular, we have ${K_v(P)=K_v(\omega)}$, where ${\omega=u(P)}$. This~$\omega$ is a root of~$f_\tau$, which implies that 
$$
|f_0(\omega)|_v= |f_0(\omega)-f_\tau(\omega)|_v<1.
$$ 
Since the leading coefficient of~$f_0$ is a $v$-adic unit, this means that ${|\omega-\alpha|_v<1}$ for some root~$\alpha$ of~$f_0$. It follows that ${|f_0'(\omega)|_v=1}$, which implies that ${|f_\tau'(\omega)|_v=1}$. Corollary~\ref{cderiv} now implies that ${K_v(\omega)=K_v(P)}$ is unramified  over~$K_v$, as wanted. \qed

\subsection{Proof of Theorem~\ref{tclose}}
Like in the proof of Theorem~\ref{tram}, one may replace~$K$ by a suitable finite extension. We will profit from it several times in this proof.

Let ${u\in K(X)}$ be such that ${K(t,u)=K(X)}$. As in the proof of Theorem~\ref{tram}, let ${F(T,U)\in K[T,U]}$ be such  that ${F(t,U)}$ is the minimal polynomial of~$u$ over $K[t]$. We denote by $R(T)$ the $U$-resultant of the polynomials~$F$ and $F'_U$. 
We claim the following.

\begin{proposition}
\label{presult}
There exists  a finite set 
${S\subset M_K}$  such that for every place ${v\in M_K\smallsetminus S}$   the following holds. Let ${P\in X(\overline{K_v})}$ be such that ${t(P)=\tau\in K_v}$ and  the field  $K_v(P)$ ramifies over~$K_v$. Then either ${|\tau|_v>1}$ or ${|R(\tau)|_v<1}$. 
\end{proposition}

\paragraph{Proof}
Write $F(T,U)$ as in~\eqref{eftu}. Then
 ${a_n(T)\mid R[T]}$ in the ring ${K[T]}$. Furthermore, there exist polynomials ${G(T,U),H(T,U)\in K[T,U]}$ such that 
\begin{equation}
\label{ebezout}
G(T,U)F(T,U)+H(T,U)F'_U(T,U)=R(T).
\end{equation}

Now let ${S\supseteq M_K^\infty}$ be a finite set of places of~$K$  such that for  every ${v\in M_K\smallsetminus S}$ the coefficients of the polynomials~$F$,~$G$ and~$H$ are $v$-adic integers, and the leading coefficient of 
$a_n(T)$ is a $v$-adic unit. 

This implies, in particular, that the coefficients of $R(T)$ are $v$-adic integers as well, and that  
\begin{equation}
\label{edivides}
\text{${a_n(T)\mid R(T)}$ in ${\OO_v[T]}$},
\end{equation}
where~$\OO_v$ is local ring of~$K_v$.

Extending the base field~$K$, we may assume that all the $(t,u)$-singular points of~$X$ are $K$-rational.

Fix ${v\in M_K\smallsetminus S}$ and let~$P$ and~$\tau$ be as in the statement of the proposition. Assume that ${|\tau|_v\le 1}$ (otherwise there is nothing to prove). Since ${K_v(P)\ne K_v}$, the point~$P$ cannot be $(t,u)$-singular, and we obtain  ${K_v(P)=K_v(\omega)}$, where ${\omega=u(P)}$.

Now we have two cases. If ${|a_n(\tau)|_v<1}$ then ${|R(\tau)|_v<1}$ by~\eqref{edivides}.

And if ${|a_n(\tau)|_v=1}$ then Corollary~\ref{cderiv} applies to the root~$\omega$ of the polynomial ${f_\tau(U)=F(\tau,U)}$.  We obtain ${|F'_U(\tau,\omega)|_v<1}$. Substituting ${T=\tau}$ and ${U=\omega}$ in~\eqref{ebezout}, we obtain ${|R(\tau)|_v<1}$. Proposition~\ref{presult} is proved. \qed

\bigskip

Extending the base field~$K$, we may assume that all roots of $R(T)$ belong to~$K$. Extending it further, we may assume that all the points from the finite set 
$$
\calA=\{P\in X(\bar K): \text{$t(P)$ is a root of $R(T)$ or $\infty$}\},
$$
are $K$-rational. 

Now let ${\tilu\in K(X)}$ be such that ${K(t,\tilu)=K(X)}$ and~$\tilu$ has pairwise distinct finite values at the points from~$\calA$. (Existence of such~$\tilu$ easily follows from the weak approximation theorem.) We define for~$\tilu$ polynomials $\tilF(T,U)$ and $\tilR(T)$ in the same way as we defined~$F$ and~$R$ for~$u$. 

\begin{proposition}
\label{pcommon}
The only common roots of $R(T)$ and $\tilR(T)$  are the finite critical values of~$t$. 
\end{proposition}

\paragraph{Proof}
Let~$\alpha$ be a root of $R(T)$ but not a critical value of~$t$. Then the fiber ${t^{-1}(\alpha)}$ consists of~$n$ distinct points. By our choice of~$\tilu$, it takes at them~$n$ distinct finite values. It follows that the polynomial $\tilF(\alpha,U)$ is of degree~$n$ and has~$n$ distinct roots. It particular, $\tilF'_U(\alpha, U)$ does not vanish at any of these roots. Hence ${\tilR(\alpha)\ne 0}$, as wanted. \qed

\bigskip

Let  ${\calC\subset K\cup\{\infty\}}$ be the set of critical values of~$t$. 

\begin{proposition}
\label{palmost}
There exists a finite set ${S\subset M_K}$ such that for every place ${v\in M_K\smallsetminus S}$   the following holds. Let ${P\in X(\overline{K_v})}$ be such that ${t(P)=\tau\in K_v}$   and  the field  $K_v(P)$ ramifies over~$K_v$. Then there exists a unique ${\alpha\in \calC\cup\{\infty\}}$  such that ${|\tau-\alpha|_v<1}$.  
\end{proposition}

\paragraph{Proof}
Applying Proposition~\ref{presult} to both~$u$ and~$\tilu$, we find~$S$ such that for every ${v\in M_K\smallsetminus S}$   the following holds. Let~$P$ and~$\tau$ be as in the statement of Proposition~\ref{palmost}. Then either ${|\tau|_v>1}$ or 
\begin{equation}
\label{enearnear}
|R(\tau)|_v<1, \qquad |\tilR(\tau)|_v<1. 
\end{equation}

Expanding~$S$, we my assume that that all the finite critical values of~$t$ are $v$-adic integers and   ${|\alpha-\alpha'|_v =1}$ for any two distinct finite critical values $\alpha$ and~$\alpha'$. 

If ${|\tau|_v>1}$ then  ${\alpha=\infty}$ is as wanted. From now on assume that ${|\tau|_v\le 1}$. 

Define ${D(T)=\gcd\bigl(R(T),\tilR(T)\bigr)}$ in the ring $K[T]$. We normalize $D(T)$ to make it monic. Proposition~\ref{pcommon} implies that that all roots of $D(T)$ are finite critical values of~$t$.
Write 
\begin{equation}
\label{edet}
D(T)=E(T)R(T)+\tilE(T)\tilR(T)
\end{equation}
with some ${E(T), \tilE(T)\in K[T]}$. Further expanding~$S$, we may assume that for ${v\in M_K\smallsetminus S}$ the coefficients of~$E$,~$\tilE$,~$R$,~$\tilR$ are $v$-adic integers. 

Substituting ${T=\tau}$ in~\eqref{edet} and using~\eqref{enearnear}, we obtain ${|D(\tau)|_v<1}$. Since~$D$ is monic, this implies that ${|\tau-\alpha|_v<1}$ for some root~$\alpha$ of~$D$, which is a critical value of~$t$. And this~$\alpha$ is unique because ${|\alpha-\alpha'|_v =1}$ for distinct critical values $\alpha$ and~$\alpha'$. \qed

\bigskip

Now we are ready to complete the proof of Theorem~\ref{tclose}. If~$\infty$ is a critical value then Proposition~\ref{palmost} does the job. Now assume that~$\infty$ is not critical. Applying Theorem~\ref{tnram} with ${\alpha=\infty}$,  a suitably expanded~$S$ has the following property: if    ${v\in M_K\smallsetminus S}$   and ${P\in X(\overline{K_v})}$ are such that ${t(P)=\tau\in K_v}$   and    $K_v(P)$ ramifies over~$K_v$, then ${|\tau|_v\le 1}$. Hence in this case Proposition~\ref{palmost} can produce only a finite~$\alpha$, which is a critical value of~$t$. \qed

\subsection{The Critical Polynomial}
\label{ssdelta}

It would be convenient to have versions of Theorems~\ref{tram} and~\ref{tclose} not assuming that the finite critical values belong to~$K$. Let ${\Delta(T)}$ be the monic separable polynomial whose roots are exactly the finite critical values of~$t$. Then, clearly, ${\Delta(T)\in K[T]}$. 

\begin{theorem}
\label{tdelta}
There exists a finite set ${S\subset M_K}$ such that for any ${v\in M_K\smallsetminus S}$ and ${\tau \in K_v}$ the following holds. 

\begin{enumerate}
\item
\label{iram}
Assume that 

\begin{itemize}
\item
either ${v\bigl(\Delta(\tau)\bigr)=1}$,

\item
or~$\infty$ is a critical value and ${v(\tau)=-1}$. 
\end{itemize}
Then~$v$ ramifies in $K_v(P)$ for some ${P\in X(\overline{K_v})}$ with ${t(P)=\tau}$.

\item
\label{iclose}
Assume that~$v$ ramifies in $K_v(P)$ for some ${P\in X(\overline{K_v})}$ with ${t(P)=\tau}$.  Then 
\begin{itemize}
\item
either ${|\Delta(\tau)|_v<1}$,

\item
or~$\infty$ is a critical value and ${|\tau|_v>1}$. 
\end{itemize}
\end{enumerate}

\end{theorem}

\paragraph{Proof}
If the statement holds true with~$K$ replaced by a finite extension~$K'$, then it holds over~$K$ as well: one only needs to exclude from consideration those finitely many places of~$K$ which ramify in~$K'$. 
This reduces the theorem to the case when all finite critical values belong to~$K$, when it becomes an immediate consequence of Theorems~\ref{tram} and~\ref{tclose}. \qed

\section{The Argument of Dvornicich and Zannier}
\label{sardz}

In this section we prove the following theorem, which is slightly stronger than Theorem~\ref{tmain}. 

\begin{theorem}
\label{tmainsize}
Let~$K$ be a number field of degree~$d$ over~$\Q$. Further, let~$X$ be a  curve over~$K$ of genus~$\bfg$ and ${t\in K(X)}$ a non-constant rational function of degree ${n\ge 2}$. 
There exist  real numbers ${c=c(K,\bfg,n)>0}$ and ${B_0=B_0(K,X,t)>1}$ such that the following holds. Pick ${P_\tau \in t^{-1}(\tau)}$ for every ${\tau \in \OO_K}$. Then for every ${B\ge B_0}$  the number field 
\begin{equation}
\label{ebigfield}
K(P_\tau:\ \tau\in \OO_K,\ \size\tau\le B)
\end{equation} 
is of degree at least $e^{cB^d/\log B}$ over~$K$. 
\end{theorem}

Theorem~\ref{tmain} is an immediate consequence because of~\eqref{esizeheight}. 
Another consequence is the following more precise version of Corollary~\ref{cmain}.

\begin{corollary}
\label{cmainsize}
Let~$K$,~$X$ and~$t$ be as in Theorem~\ref{tmainsize}. Then there exist  real numbers ${c=c(K,\bfg,n)>0}$ and ${B_0=B_0(K,X,t)>1}$ such that the following holds. Pick ${P_\tau \in t^{-1}(\tau)}$ for every ${\tau \in \OO_K}$. Then for every ${B\ge B_0}$, among the number fields 
$$
K(P_\tau) \qquad (\tau\in \OO_K,\quad \size\tau\le B)
$$ 
there are  at least $cB^d/\log B$ distinct fields.
\end{corollary}

In the sequel~$K$,~$X$ and~$t$ as in the statement of Theorem~\ref{tmainsize}. Everywhere in this section we adopt the following conventions. 

\begin{itemize}

\item
``Sufficiently large'' means ``exceeding a certain quantity depending on~$K$,~$X$ and~$t$'';

\item
``Almost every'' means ``outside a finite set   depending on~$K$,~$X$ and~$t$''.

\end{itemize}

We will adapt the beautiful ramification argument of Dvornicich and Zannier~\cite{DZ94}, which, as they remark, traces back to the work of Davenport, Lewis and Schinzel~\cite{DLS64}. We refer to Section~3 of~\cite{BL16} for a concise exposition of the Dvornicich-Zannier argument over~$\Q$. 

Let~``$\prec$'' be a strict order relation on a countable set~$M$. We call it  \textsl{$\N$-ordering} if $(M,\prec)$ and $(\N,<)$ are isomorphic as ordered sets. 
We fix an $\N$-ordering~``$\prec$'' on the set~$\OO_K$ which makes the size function (non-strictly) increasing: for ${\beta, \beta'\in \OO_K}$ with ${\beta\prec\beta'}$ we have ${\size\beta\le\size{\beta'}}$.

We say that a place ${v\in M_K}$ is \textsl{primitive} for ${\tau \in \OO_K}$  if~$v$ ramifies in the field extension ${K\bigl(t^{-1}(\tau)\bigr)/K}$, but does not ramify in ${K\bigl(t^{-1}(\tau')\bigr)/K}$ for any ${\tau'\prec\tau}$.

Theorem~\ref{tmainsize} is a consequence of the following two statements. 

\begin{proposition}
\label{pprim}
Let~$\alpha$ be a finite critical value of~$t$. Then 
almost every ${v\in M_K}$ having an extension  ${w \in M_{K(\alpha)}}$ of degree ${[w:v]=1}$ serves as primitive for some ${\tau \in \OO_K}$ satisfying ${\size\tau\le\lambda
(\norm v)^{1/d}}$. Here ${\lambda\ge 1}$ depends only on~$K$. 
\end{proposition}

Recall that $\norm v$ denotes the $K/\Q$-norm of the prime ideal of~$v$, with the convention ${\norm v=1}$ for an infinite place~$v$. 

\begin{proposition}
\label{pfew}
Let~$m$ be the total number of  finite critical values of~$t$. Let~$\eps$ be a real number satisfying ${0<\eps\le 1}$. 
For  every ${\tau\in\OO_K}$ with ${\size\tau\ge \kappa\eps^{-m-1}}$ (where ${\kappa\ge 1}$ depends only on~$X$ and~$t$) there is at most~$m$ finite places~$v$ of~$K$ satisfying ${\norm v\ge \bigl(\eps \size\tau\bigr)^d}$ and ramified in $K(t^{-1}(\tau))$.
\end{proposition}

\subsection{Proof of Theorem~\ref{tmainsize}} 
In this subsection we prove Theorem~\ref{tmainsize} assuming validity of Propositions~\ref{pprim} and~\ref{pfew}.

Let~$\alpha$ be a finite critical value of~$t$ (which exists because ${n\ge 2}$). 
Denote by~$\Mu$ the set of ${v\in M_K}$ satisfying the hypothesis of Proposition~\ref{pprim}. The Tchebotarev Density Theorem implies that that there exists ${\delta>0}$ such that 
\begin{equation}
\label{etchebo}
\bigl|\{v\in \Mu: \norm v\le B\}\bigl|\sim \delta \frac{B}{\log B}
\end{equation}
as ${B\to\infty}$. This~$\delta$ can be estimated from below only in terms of ${\mu=[K(\alpha):K]}$; in fact, it is easy to see that 
\begin{equation}
\label{elowerdelta}
\delta\ge \frac1\mu\ge \frac1m,
\end{equation} 
where~$m$ is the total number of finite critical values.

Now, for a given ${B\ge 1}$ we define the following three sets:
\begin{align*}
\Mu(B)&=\left\{v\in \Mu: \left(\frac B{2\lambda}\right)^d\le \norm v\le \left(\frac B{\lambda}\right)^d\right\},\\
\Omega(B)&=\{\tau \in \OO_K: \text{$\tau$ has a primitive $v\in M(B)$}\},\\
\Omega'(B)&=\{\tau\in \Omega(B): \text{the fiber ${t^{-1}(\tau)}$ is $K$-irreducible}\}. 
\end{align*}
Proposition~\ref{pprim} implies that 
\begin{equation}
\label{etauleb}
\text{${\size\tau\le B}$ for every ${\tau \in \Omega(B)}$}. 
\end{equation}

If~$\tau$ admits a primitive~$v$ then 
 the field 
${K\bigl(t^{-1}(\tau)\bigr)}$  is not contained in the compositum of all the ``preceding'' fields $K\bigl(t^{-1}(\tau')\bigr)$ where  ${\tau'\prec \tau}$.
If, in addition to this, the fiber ${t^{-1}(\tau)}$ is $K$-irreducible, then the field $K\bigl(t^{-1}(\tau)\bigr)$ is the Galois closure (over~$K$) of $K(P_\tau)$, which implies that 
$K(P_\tau)$ is not contained in the compositum of the ``preceding'' fields 
$K(P_{\tau'})$ with ${\tau'\prec \tau}$.
Combining this with~\eqref{etauleb}, we conclude that the degree of the field~\eqref{ebigfield} over~$K$ is at least ${2^{|\Omega'(B)|}}$. We are left with proving that 
\begin{equation}
\label{eloweromega}
|\Omega'(B)|\ge c\frac{B^d}{\log B}, 
\end{equation}
where ${c>0}$ depends on~$K$,~$\bfg$ and~$n$.

Using~\eqref{etchebo} and~\eqref{elowerdelta}, we estimate 
$$
|\Mu(B)|\ge \frac\delta{2d\lambda^d}\frac{B^d}{\log B} \ge \frac1{2md\lambda^d}\frac{B^d}{\log B}
$$
for sufficiently large~$B$. 
Now Proposition~\ref{pfew} applied with ${\eps=(2\lambda)^{-1}}$ implies that, for sufficiently large~$B$ 
$$
|\Omega(B)|\ge \frac1m|\Mu(B)|\ge \frac1{2m^2d\lambda^d}\frac{B^d}{\log B}. 
$$
Further, Corollary~\ref{chilb} implies that 
$$
|\Omega'(B)|\ge |\Omega(B)|-O(B^{d/2})\ge \frac1{4m^2d\lambda^d}\frac{B^d}{\log B}. 
$$
for sufficiently large~$B$.

Finally, the Riemann-Hurwitz formula implies that ${m\le 2\bfg+2n-2}$. This proves~\eqref{eloweromega} with 
$$
c= \frac1{16(\bfg+n)^2d\lambda^d},
$$
as wanted.\qed

\subsection{Proof of Proposition~\ref{pprim}}

In this subsection we prove Proposition~\ref{pprim}. 
Let~$\alpha$,~$v$ and~$w$ be as in the statement of the proposition. Throwing away finitely many~$v$  we may assume that ${ w(\alpha)\ge 0}$. Since ${[w:v]=1}$, there exists ${\alpha'\in \OO_K}$ such that ${ w(\alpha-\alpha')>0}$. Corollaries~\ref{cmod} and~\ref{cprim} imply that there exist ${\gamma',\pi\in \OO_K}$ such that 
$$
 v(\gamma'-\alpha')>0, \quad  v(\pi) =1, \quad \size{\gamma'},\size\pi  \ll(\norm v)^{1/d}, 
$$
where in this proof the constants implied by the $O(\cdot)$-notation and by the Vinogradov $\ll$-notation depend only on~$K$. 

Now set 
$$
\gamma=
\begin{cases}
\gamma', & w(\gamma'-\alpha)=1,\\
\gamma'+\pi,&  w(\gamma'-\alpha)>1.
\end{cases}
$$
Then ${ w(\gamma-\alpha)=1}$ and ${\size{\gamma}\ll(\norm v)^{1/d}}$. 

Theorem~\ref{tram} (applied to the field $K(\alpha)$ instead of~$K$) implies that, unless our~$w$ belongs to a finite exceptional set, it ramifies in the field ${K\bigl(\alpha, t^{-1}(\gamma)\bigr)}$. It follows that~$v$ ramifies in the latter field as well. This means that~$v$ ramifies either in~$K(\alpha)$ (which is the case for only finitely many~$v$) or in ${K\bigl(t^{-1}(\gamma)\bigr)}$.

We have proved the following: for almost every ${v\in M_K}$ satisfying the hypothesis of the proposition, the set 
$$
\{\gamma\in \OO_K:  \text{$v$ ramifies in  ${K(t^{-1}(\gamma)}$}\}
$$
is non-empty and contains an element of size $O\bigl((\norm v)^{1/d}\bigr)$. Taking as~$\tau$ the smallest element of this set with respect to the ``$\prec$'' ordering, we complete the proof. \qed

\subsection{Proof of Proposition~\ref{pfew}}

If ${\norm v\ge \bigl(\eps\size\tau\bigr)^d}$ and ${\size\tau\ge \kappa\eps^{-m-1}}$ then 
$$
\norm v\ge\kappa^d\eps^{-md}\ge \kappa.
$$
Selecting~$\kappa$ sufficiently large, we may assume that the finitely many exceptional places from Theorem~\ref{tdelta}:\ref{iclose} are all of norm smaller than~$\kappa$. Hence 
we only have to count ${v\in M_K}$ satisfying 
\begin{equation}
\label{eimpossible}
|\Delta(\tau)|_v<1, \qquad \norm v\ge \bigl(\eps\size{\tau}\bigr)^d, 
\end{equation}
where $\Delta(T)$ is the ``critical polynomial'' from  Subsection~\ref{ssdelta}. Assuming that there is ${m+1}$ such places, we will show that ${\size\tau\ll \eps^{-m-1}}$,   where in this proof the constants implied by the $O(\cdot)$-notation and by the Vinogradov $\ll$-notation depend only on~$X$ and~$t$. 

Thus, let ${v_1, \ldots, v_{m+1}}$ be distinct places of~$K$ such that every ${v=v_i}$ satisfies~\eqref{eimpossible}. We denote by~$\norm$ the $K/\Q$-norm. Then $\norm v_1\cdots\norm v_{m+1}$ divides the numerator of the rational number $\norm\Delta(\tau)$. Since ${\deg\Delta=m}$, we have 
${|\norm\Delta(\tau)|\ll{\size\tau}^{md}}$, 
and the denominator of this number is $O(1)$. Hence 
$$
\norm v_1\cdots\norm v_{m+1}\ll{\size\tau}^{md}.
$$
On the other hand, the left-hand side is bounded from below by ${\bigl(\eps\size\tau\bigr)^{d(m+1)}}$. Comparing the lower and the upper bound, we obtain ${\size\tau\ll \eps^{-m-1}}$, as wanted. \qed

\renewcommand\thesection{A}

\section{Appendix (by Jean Gillibert)}

The aim of this appendix is to give a scheme-theoretic explanation of the relation between geometric and arithmetic ramification of covers of $\P^1$. More precisely, we shall give alternative statements and proofs for Theorems~\ref{tram},~\ref{tnram} and~\ref{tclose} of Section~\ref{sram}. We also give both a conceptual and explicit description of the set of ``bad'' places involved in these statements.

The results that we prove here are essentially due to Beckmann~\cite{Be91}. They have been subsequently generalized by Conrad~\cite{Co00}. However, it may be useful for the reader to have a short proof of the statements we are interested in. As we shall see, the main technical tool is Abhyankar's Lemma.

Let us recall the notation:~$K$ is a number field,~$X$ is a smooth projective curve over~$K$, and ${t:X\to \P^1}$ is a non-constant $K$-morphism. Then~$t$ is a finite map, and is {\'e}tale outside a divisor ${D\subset\P^1}$, that one calls the branch locus of~$t$.

Let ${\XX\to \P^1_{\OO_K}}$ be the normalization of $\P^1_{\OO_K}$ in the function field of~$X$ (via the map~$t$). The canonical map ${\XX\to \P^1_{\OO_K}}$ is a finite flat map with generic fiber~$t$, that we also denote by~$t$ by abuse of notation. Let ${\DD\subset \P^1}$ be the branch locus of ${t:\XX\to\P^1_{\OO_K}}$. Then, the scheme ${\P^1_{\OO_K}}$ being regular and the scheme $\XX$ being normal, the subscheme~$\DD$ is of pure codimension one, according to the Zariski-Nagata purity Theorem for the branch locus (See \cite[expos{\'e}~X, Theorem.~3.1]{sga1}). In other terms,~$\DD$ is a divisor on $\P^1_{\mathcal{O}_K}$.

Let $\overline{D}$ be the scheme-theoretic closure of $D$ in $\P^1_{\mathcal{O}_K}$. This is a horizontal divisor on $\P^1_{\mathcal{O}_K}$, with generic fiber $D$. Therefore, one can write
$$
\mathcal{D}=\overline{D}+V
$$
where $V$ is a vertical divisor (hence supported by a finite number of fibers). Let $S$ be the union of the following two finite sets of finite places of $K$:\footnote{This set $S$ is similar to that introduced by Beckmann. We note that Conrad considers a smaller set $S$ by allowing two horizontal components of $\mathcal{D}$ to intersect with multiplicity one, but for simplicity we stick to this definition.}
\begin{enumerate}
\item[(i)] the set of places supporting $V$,
\item[(ii)] the set of places above which two branch points of $t$ meet, or above which the field of definition of a branch point is ramified.
\end{enumerate}

By (i), the equality $\mathcal{D}=\overline{D}$ holds true in $\P^1_{\mathcal{O}_{K,S}}$. By (ii), $\overline{D}$ is a geometrically unibranch\footnote{Let us recall that a divisor is geometrically unibranch if its irreducible components are disjoint, and if this remains true after any {\'e}tale base change. For example, the nodal cubic $y^2=x^2(x+1)$ is an irreducible normal crossing divisor in $\P^2$, but is not geometrically unibranch.} divisor over $\P^1_{\mathcal{O}_{K,S}}$. In particular, it has strict normal crossings, and disjoint irreducible components.

\begin{remark}
\label{rmkA1}
If the curve $X$ and the map $t$ are explicitely given (by equations), then it is possible to compute the set $S$, without computing the integral model $\XX$. Indeed, the set of places supporting $V$ is just the set of $v$ for which the field extension $K(X)/K(\P^1)$ corresponding to $t$ is ramified at $v$, where $v$ is viewed as a valuation on $K(\P^1)$. The other piece of the set $S$ depends only on the knowledge of the set of branch points, and is defined in quite explicit terms.
\end{remark}

The reason for which we introduce the set $S$ lies in the following Lemma:

\begin{lemma}
\label{lemma1}
The map $t:\XX\to\P^1_{\mathcal{O}_{K,S}}$ is a tame cover with respect to the normal crossing divisor $\mathcal{D}$, in the sense of Grothendieck and Murre \cite[Definition~2.2.2]{GM71}. More precisely:
\begin{enumerate}
\item[1)] $t$ is finite,
\item[2)] $t$ is {\'e}tale outside $\mathcal{D}$,
\item[3)] every irreducible component of $\XX$ dominates an irreducible component of $\P^1_{\mathcal{O}_{K,S}}$,
\item[4)] $\XX$ is normal,
\item[5)] $\XX$ is tamely ramified above each $x\in \mathcal{D}$ of codimension $1$ in $\P^1_{\mathcal{O}_{K,S}}$.
\end{enumerate}
\end{lemma}

\begin{proof}
As we have seen above, 1) is true by construction of $t$, and 2) by definition of $\mathcal{D}$. The map $t$ is surjective, hence 3) holds. The scheme $\XX$ is normal by construction, so 4) is OK. It just remains to check 5). The local fields of codimension $1$ points of $\mathcal{D}$ have characteristic zero, because $\mathcal{D}=\overline{D}$, therefore the tameness assumption is automatically satisfied.
\end{proof}

We now recover the nice framework of the theory of tame covers by Grothen\-dieck and Murre. It is certainly possible to transpose these results in the language of log schemes, but we do not need such techniques for our purpose.

We are now ready to state the main result of this appendix.

\begin{theorem}
\label{appendix_theorem}
Let $v\notin S$, and let $P\in X(\bar{K_v})$ which is not a ramification point of $t$. Then the following holds:
\begin{enumerate}
\item[1)] If the extension $K_v(P)/K_v(t(P))$ is ramified, then there exists a branch point $\alpha\in\P^1(\bar{K_v})$ such that $v(t(P)-\alpha)\geq 1$.
\item[2)] If there exists a branch point $\alpha\in\P^1(\bar{K_v})$ such that $v(t(P)-\alpha)= 1$, then the extension $K_v(P)/K_v(t(P))$ is ramified.
\item[3)] More generally, if there exists a branch point $\alpha\in\P^1(\bar{K_v})$ such that $v(t(P)-\alpha)$ is strictly positive and not divisible by any of the (geometric) ramification indices of points in $t^{-1}(\alpha)$, then the extension $K_v(P)/K_v(t(P))$ is ramified.
\end{enumerate}
\end{theorem}

In the statement above, we use the standard convention ${t(P)-\infty=t(P)^{-1}}$, as in the main text of the article.
We note that ${v(t(P)-\alpha)\geq 1}$ means that $t(P)$ and~$\alpha$ reduce to the same point in $\P^1(\bar{k_v})$, where $k_v$ denotes the residue field of $K_v$. In geometric language,  ${v(t(P)-\alpha)}$ is the intersection number between $t(P)$ and~$\alpha$.

Roughly speaking, Theorem~\ref{appendix_theorem} states that, outside the finite set $S$, the arithmetic ramification is controlled by the reduction of the geometric ramification.
The statement 1) above implies Theorems~\ref{tnram} and~\ref{tclose}, while statement 2) implies (via Hensel's lemma) Theorem~\ref{tram}. Our hypotheses are slightly weaker than in Section~\ref{sram}, in particular  branch points of~$t$ are not assumed to be rational over the base field.

First, we state a version of Abhyankar's Lemma, suitable for our purpose:

\begin{lemma}[Abhyankar's Lemma]
\label{abhyankar}
Let $Y$ be a noetherian normal scheme, and let $g:X\to Y$ be a tame cover with respect to a normal crossing divisor $D\subset Y$. Assume furthermore that $D$ is geometrically unibranch. Then for each $y\in \Supp(D)$ there exists an \'etale neighbourhood $U\to Y$ of $y$ such that $X\times_Y U\to U$ is a finite disjoint union of coverings of the form
$$
\Spec(\mathcal{O}_U[T]/(T^e-f))\longrightarrow U
$$
where $f=0$ is a local equation of $D$ in $U$, and $e$ is relatively prime to the characteristic of the residue field of $y$.
\end{lemma}

\begin{proof}
This follows from  \cite[Corollary~2.3.4, ii)]{GM71}. More precisely, $D$ being geometrically unibranch, the irreducible components of $D$ remain disjoint over any {\'e}tale neighbourhood of $y$ in $Y$. Hence the last sentence in the statement of \cite[Corollary~2.3.4, ii)]{GM71} implies that there exists an \'etale neighbourhood ${U\to Y}$ of~$y$ such that ${X\times_Y U\to U}$ is a finite disjoint union of Kummer coverings\footnote{See \cite[Definition~1.2.2]{GM71} for a definition.} with respect to the divisor $D\times_Y U$. \qed
\end{proof}

\begin{remark}
\label{rmkA2}
Under the hypoteses of Lemma~\ref{abhyankar}, $X$ is regular above the points of $\Supp(D)$, according to \cite[Proposition~1.8.5, ii)]{GM71}. In particular, if $Y$ is regular, then $X$ is regular. We note that this is no longer true if $D$ is an arbitrary normal crossing divisor: in this case, one can describe $X$ {\'e}tale locally as a \emph{generalized} Kummer covering\footnote{That is, a quotient of a Kummer covering, see \cite[Definition~1.3.8]{GM71}.} of $Y$. According to \cite[Proposition~1.8.5 iii)]{GM71}, such a covering may be singular above the intersection of two irreducible components of $D$. See also \cite[Lemma~1.4]{Co00} and the erratum.
\end{remark}

\paragraph{Proof of Theorem~\ref{appendix_theorem}}
1) Let $\mathcal{O}_v$ be the ring of integers of $K_v$. By projectivity of $\XX$ the point $P$ extends into a section $P:\Spec(\mathcal{O}_v(P))\to \XX$ where $\mathcal{O}_v(P)$ is the ring of integers of $K_v(P)$. By definition of the branch locus, if  $t(P)$  does not meets $\mathcal{D}$ (above $v$), then $\mathcal{O}_v(P)$ is an {\'e}tale $\mathcal{O}_v(t(P))$-algebra, that is, $K_v(P)/K_v(t(P))$ is unramified. By construction of $S$, the divisor $\mathcal{D}$ is the scheme-theoretic closure of $D$ in $\P^1_{\mathcal{O}_{K,S}}$, hence for $v\not\in S$ the subschemes $t(P)$ and $\mathcal{D}$ have non-empty intersection above $v$ if and only if there exists a branch point $\alpha\in \P^1(\bar{K_v})$ such that $t(P)$ and $\alpha$ reduce to the same point in $\P^1(\bar{k_v})$.

2) This is a special case of 3).

3) In order to prove the statement, we may assume that $t(P)$ belongs to $K_v$. Let $\mathcal{O}_v^{\sh}$ be the strict henselization of the ring of integers of $K_v$ with respect to the valuation $v$. Then the fraction field of $\mathcal{O}_v^{\sh}$ is $K_v^{\nr}$, the maximal unramified extension of $K_v$.

Let $\pi_v$ be a uniformizing parameter of $\mathcal{O}_v$, and let $k_v$ be the residue field of $\mathcal{O}_v$. Then $\pi_v$ is again a uniformizing parameter of $\mathcal{O}_v^{\sh}$, and the residue field of $\mathcal{O}_v^{\sh}$ is $\bar{k_v}$.


Let $\alpha\in\P^1(\bar{K_v})$ be a branch point of $t$ such that $v(t(P)-\alpha)>0$. Then, by construction of the set $S$, $\alpha$ belongs to $\P^1(K_v^{\nr})$. Up to composing $t$ by an automorphism of $\P^1$ defined over $K_v^{\nr}$, it is possible to assume that $\alpha=0$. Let $z$ be the standard coordinate function on $\P^1_{\mathcal{O}_v^{\sh}}$ which vanishes at $0$. Let $\bar{0}$ be the closed point of $\P^1_{\mathcal{O}_v^{\sh}}$ defined by $z=0$ on the special fiber.
Then $z=0$ is a local equation of $\mathcal{D}$ at the point $\bar{0}$ (this follows from the fact that $\mathcal{D}$ is horizontal with disjoint components). Moreover, the ring $\mathcal{O}_v^{\sh}[[z]]$ of formal power series is strictly henselian, because it is complete with respect to the $(z,\pi_v)$-adic topology and its residue field is $\bar{k_v}$ which is algebraically closed. We have a natural ``localization'' map $\Spec(\mathcal{O}_v^{\sh}[[z]])\to \P^1_{\mathcal{O}_v^{\sh}}$ which sends the closed point to $\bar{0}$.

According to Lemma~\ref{lemma1}, the map $t:\XX\to \P^1_{\mathcal{O}_{K,S}}$ is a tame cover with respect to the normal crossing divisor $\mathcal{D}$. Hence, it follows from \cite[Cor.~2.3.6]{GM71} that the pull-back
$$
\XX\times_{\P^1} \Spec(\mathcal{O}_v^{\sh}[[z]])\longrightarrow \Spec(\mathcal{O}_v^{\sh}[[z]])
$$
is a tame cover with respect to the divisor $\{z=0\}$.

Hence, according to Lemma~\ref{abhyankar} (Abhyankar's Lemma), this scheme is a disjoint union of connected Kummer coverings of the form
$$
\Spec(\mathcal{O}_v^{\sh}[[z]][T]/(T^e-z))\to \Spec(\mathcal{O}_v^{\sh}[[z]])
$$
where $e\geq 2$ is the ramification index of some point $\beta\in X$ lying above $\alpha$. Let us consider the connected cover containing the point $P$. We may specialize it at the integral section $t(P)$, which gives us the finite cover
$$
\Spec(\mathcal{O}_v^{\sh}[T]/(T^e-t(P))) \longrightarrow \Spec(\mathcal{O}_v^{\sh}).
$$
It follows that the field $K_v^{\nr}(P)$ contains at least one root of the polynomial $T^e-t(P)$. Such a field is a ramified extension of $K_v^{\nr}$ if and only if $e$ does not divides $v(t(P))$. Hence the result. \qed

\begin{remark}
\label{rmkA3}
According to the last part of the proof above, if there exists a branch point $\alpha$ such that $v(t(P)-\alpha)\geq 1$, then there exists a unique ramification point $\beta$ lying above $\alpha$ which meets $P$ above $v$. Moreover, if $v(t(P)-\alpha)$ is divisible by the ramification index of $\beta$, then the extension $K_v(P)/K_v(t(P))$ is unramified, according to the last sentence of the proof.
\end{remark}

\begin{example}
Let $\lambda\in \Q$, distinct from $0$ and $1$, and let $E$ be the elliptic curve defined over $\mathbb{Q}$ by the equation (in Legendre form)
$$
y^2=x(x-1)(x-\lambda).
$$
We consider the $x$-coordinate map $x:E\to\P^1$, which is a degree $2$ cover, and whose branch points are $0$, $1$, $\lambda$ and $\infty$. At the function field level, the corresponding extension is
$$
\mathbb{Q}\left(\sqrt{x(x-1)(x-\lambda)}\right)/\mathbb{Q}(x).
$$
We note that the only prime number which ramifies in this extension is $2$. This means that the vertical ramification is supported by $2$.

We are now looking for the set of primes above which two branch points meet. We note that the sections $0$, $1$ and $\infty$ never meet each other in $\P^1(\Z)$. Let $\lambda=\frac{a}{b}$ where $a$ and $b$ are coprime integers. Then $\lambda$ meets $0$ (resp. $\infty$) above primes dividing $a$ (resp. $b$). Similarly, $\lambda$ meets $1$ above primes dividing $a-b$. Therefore, the set $S$ is:
$$
S = \{2\}\cup \{\text{primes dividing $ab(a-b)$}\}.
$$
In the light of Remark~\ref{rmkA3}, our Theorem~\ref{appendix_theorem} reads as follows: let $p\not\in S$ be a prime number, and let $P\in E(\bar{\mathbb{Q}_p})$ which is not a ramification point of $x$. Then the extension
$$
\mathbb{Q}_p(P)=\mathbb{Q}_p\left(\sqrt{x(P)(x(P)-1)(x(P)-\lambda)}\right)/\mathbb{Q}_p(x(P))
$$
is ramified if and only if there exists a branch point $\alpha\in\{0, 1, \lambda, \infty\}$ such that $v_p(x(P)-\alpha)$ is strictly positive and odd.
\end{example}



\end{document}